\documentclass[3pt]{article}

\usepackage{amsmath,amssymb,amsfonts,bbm,amsthm,color,pdfpages,graphicx,geometry,enumerate}
\usepackage{extarrows,booktabs,multirow,url}

\theoremstyle{plain}
\newtheorem{thm}{\bf Theorem}[section]

\newtheorem{lem}[thm]{\bf Lemma}

\newtheorem{defn}[thm]{\bf Definition}
\theoremstyle{remark} \newtheorem{remark}[thm]{\bf Remark}

\newcommand{\Rg}       {{\hbox{I\kern-.22em\hbox{R}}}}
\newcommand{\Pg}       {{\hbox{I\kern-.22em\hbox{P}}}}
\newcommand{\Eg}       {{\hbox{I\kern-.22em\hbox{E}}}}

\newcommand{\cov}{\mathop{\mathrm{cov}}\nolimits}

\title{LAN Property for the Drift and Hurst Paramters in The Mixed Fractional O-U Process with Continuous Observations }
\author{Cai Chunhao and Zhang Cong\\Sun Yat Sen University and Zhejiang University \\chunhao\_cai@nankai.edu.cn\\zcong@zju.edu.cn}
\date{\today}

\begin{document}
\maketitle

\begin{abstract}
This paper deals with the Local Asymptotical normality for the joint drift parameter and Hurst parameter $H>3/4$ in the mixed fractional Ornstein-Uhlenbeck process. Different from the only estimation of the drift parameter when $H$ is known, we will use the fact that the mixed fractional Brownian motion is a semimartingale with its own filtering when $H>3/4$.  \\

\end{abstract}
\section{Introduction}
The Ornstein-Uhlenbeck process has been extensively applied in various fields, as diverse as economics, finance, high technology, biology, physics, chemistry, medicine and environmental studies. In fact, the standard Ornstein-Uhlenbeck models, including the diffusion models based on the Brownian motion and the jump-diffusion models driven by L\'{e}vy processes, provide good service in cases where the data demonstrate the Markovian property and the lack of memory. However, over the past few decades, numerous empirical studies have found that the phenomenon of long-range dependence may observe in data of hydrology, geophysics, climatology and telecommunication, economics and finance. In the continuous time situation, the best known and widely used stochastic process that exhibits long-range dependence is of course the fractional Brownian motion (fBm), which describes the degree of dependence by the Hurst parameter. Consequently, fBm is the usual candidate to capture some properties of ``Real-world" data and the applications of the fractional Ornstein-Uhlenbeck process (FOUP) have recently experienced intensive development (see, for example, \cite{comte1998long,CCR12}). The development of the application for fBm and FOUP naturally led to the statistical inference for stochastic models driven by fBm. Consequently, the parameter estimation problems for stochastic models driven by fBm and FOUP have been of great interest in the past decade, and beside being a challenging theoretical problem. Some surveys and complete literatures related to the parametric and other inference procedures for stochastic models driven by fBm could be found in \cite{Nualart,kubilius,HNZ19}.

The mixed fractional Brownian motion is the basic noise model in engineering applications, such as astronomical data \cite{Xu17}, GPS communications \cite{MHD99}. In finance, it has a very important property:  when $H>3/4$ we can exclude all arbitrage strategies and details can be found in \cite{cheridito}. From this view, the O-U process driven by mixed fractional Brownian motion will also play a very important role in the financial market and we should focus on the parametric estimation for this model. Now let us define 
\begin{equation}\label{eq: mixed OU}
dX_t=-\alpha X_tdt+dM_t^H,\, t\in [0,T],\,\alpha>0,\,  X_0=0.
\end{equation}
where $M_t^H=B_t^H+W_t$, the $B^H=\left(B_t^H,\, t\in [0,T]\right)$ is the fractional Brownian motion motion with Hurst parameter $H$ and $W$ is the independent standard Brownian motion. 
\begin{remark}
Here the initial $X_0$ can be any real number or random variable, we can add also the variace before the fractional Brownian motion and the standard Brownian motion, but they will not affect the final result and the procedure of the proof.  
\end{remark}
In \cite{CK18}, the authors consider when $H\in (0,1)$ is known, how to construct the MLE for the drift parameter from continuous observation datas and its asymptotical properties. However, when $H$ is unknown, this method will not be useful because the observation will multiply a kernel function containing the unknown parameter $H$ and this new process can not be considered as the data. In order to avoid this problem, in this paper we will use the fact that when $H>3/4$, the mixed fractional Brownian motion  is a semimartingale with respect to its own filtering which has been presented in \cite{CCK16}.  The details is that: when $H>3/4$
\begin{equation}\label{eq: BM}
\overline{B}_t = M_t^H- \int_{0}^{t}\rho_t(M; H)\,dt, t\in [0,T]
\end{equation}
is a Brownian motion with respect to $\mathcal{F}_t^M = \sigma\left\{M_s^H, s \le t\right\}$. Here 
\begin{equation}\label{eq: rho}
\rho_t(M; H) = \int_{0}^{t}g(t, t-s; H)\,dM_s^H, 
\end{equation}
where the function $g(t,s:H)$ uniquely solves the integral equation
\begin{equation}\label{eq: g}
g(t,s;H) + \int_{0}^{t}K_H(r-s)g(t,r;H)\,dr = K_H(s), 0<s<t, 
\end{equation}
and 
\begin{equation}\label{eq: K}
K_{H}(\tau)=H(2H-1)\left\lvert \tau \right\rvert ^{2H-2}
\end{equation}
which defines the covariance of fractional Brownian motion:
$$
\cov(B_s^H,\, B_t^H)=\int_0^s \int_0^t K_{H}(u-v)dudv.
$$

From the sample $X^T=\left(X_t,\, t\in [0,T]\right)$ Our objective is to identify the best achievable minimax rates in the large time ($T\rightarrow \infty$) asymptotic regimes for the unknown parameter $\vartheta=(\alpha, H)$. To this end, we prove the Local Asymptotic Normality (LAN) property and discuss the construction of the rate optimal estimators. 

\section{Main Result}
\subsection{The LAN Property and H\'aj\`ek Lower Bound}
First let us briefly introduce Le Cam's LAN property and its role in the asymptotic theory of estimation. An abstract parametric statistical experiment consists of a measurable space $\left(\Omega,\, \mathcal{F}\right)$, where $\mathcal{F}$ is a $\sigma$-algebra of subsets of $\Omega$, a family of probability measures $(\mathbf{P}_{\theta})_{\theta\in \Theta}$ on $\mathcal{F}$ with the parameter space $\Theta \subset \mathbb{R}^k$ and the sample $X \sim \mathbf{P}_{\theta_0}$
for a true value $\theta_0\in \Theta$ of the parameter variable. Asymptotic theory is concerned with a family of statistical experiments $\left(\Omega^T,\, \mathcal{F}^T,\, (\mathbf{P}_{\theta}^T)_{\theta\in \Theta}\right)$ indexed by a real valued variable $T>0$.

\begin{defn}
A family of probability measures $\left(\mathbf{P}_{\theta}^T\right)_{\theta \in \Theta}$ is LAN at a point $\theta_0$ as $T\rightarrow \infty$ if there exists a family of nonsingular $k\times k$ matrix $\phi(T)=\phi(T,\, \theta_0)$, such that for any $u\in \mathbb{R}^k$
$$
\log \frac{d\mathbf{P}^T_{\theta_0+\phi(T)u}}{d\mathbf{P}^T_{\theta_0}} \left(X^T\right)=u^*Z_{T,\theta_0}-\frac{1}{2}\| u \|^2+r_T(u,\theta_0),
$$
where
\begin{itemize}
\item $Z_{T,\theta_0}$ converges weakly under $\mathbf{P}^T_{\theta_0}$ to the standard normal law on $\mathbb{R}^k$

\item $r_T(u,\theta_0)$ vanishes in $\mathbf{P}^T_{\theta_0}$--probability as $T\rightarrow \infty$.
\end{itemize}

\end{defn}

Define a set $W_{2,k}$ of loss functions $\ell: \mathbb{R}^k \rightarrow \mathbb{R}$, which are continuous and symmetric with $\ell(0)=0$, have convex sub-level sets  $\{u:\ell(u)<c \}$ for all $c>0$
and satisfy the growth condition $\lim_{\|u\|\rightarrow 0}\exp(-a\|u \|^2)\ell(u)=0,\, \forall a>0$.  The
following theorem establishes asymptotic lower bound for the corresponding local minimax risks of estimators in LAN families.
\begin{thm}
Let $\left(\mathbf{P}_{\theta}^T\right)_{\theta\in \Theta}$ satisfy the LAN property at $\theta_0$ with matrix $\phi(T,\theta_0)\rightarrow 0$ as $T\rightarrow \infty$. Then for any family of estimator $\hat{\theta}_T$, a loss function $\ell\in W_{2,k}$ and any $\delta>0$ we have 
$$
\lim\limits_{\overline{T\rightarrow \infty}}\sup\limits_{\|\theta-\theta_0\|<\delta}\mathbf{E}_{\theta}^T\ell^{-1}(\phi(T,\theta_0)(\hat{\theta}_T-\theta)) \geq \int_{\mathbb{R}^k}\ell(x)\gamma(x)dx,
$$
where $\gamma$ is the standard normal density on $\mathbb{R}^k$.
\end{thm}
Proof  for this theorem can be found in  Theorem 12.1 of \cite{IH1981}. Estimators which achieve H\'aj\`ek lower bound are called asymptotically efficient in the local minimax sense. Usually likelihood based estimators, such as the Maximum Likelihood or the Bayes estimators with positive prior densities,
are asymptotically eﬃcient. However, they can be excessively complicated and thus one-step or multi-step MLE can be considered such as presented in \cite{BBCS24} or Whittle Likelihood Estimator \cite{Whittle53, Whittle62}. 

\begin{remark}
In fact, we usually consider the LAN property with discrete data in the fractional noise such as the high-frequency case \cite{BF18}. The continuous data case can be found in \cite{CK18, LNT19, Chiba2018}.
\end{remark}

\subsection{Main Theorem: LAN for mixed O-U process}
Denote by $\mathbf{P}_{\theta}^T$ the probability measure on the space of continuous functions
$C[0,T],\, \mathbb{R}$ induced by the mixed O-U process $X^T$  defined in \eqref{eq: mixed OU} with parameter $\theta=(H, \alpha)\in (3/4,1)\times \mathbb{R}_+$,  we have 
\begin{thm}\label{th: LAN for OU}
The family $(\mathbf{P}^T_\theta)_{\theta \in \Theta}$ is LAN at any $\theta_0 \in \Theta$ as $T \to \infty$ with
$\phi(T) = T^{-\frac{1}{2}}I(\theta_0)^{-\frac{1}{2}}$, where 
\begin{equation}\label{eq: fisher matrix}
I(\theta) = 
\begin{pmatrix}
\frac{1}{4\pi}\int_{-\infty}^{\infty}(\partial_H \log(\hat{K}_H(\lambda)+1))^2 \,d\lambda& \frac{\alpha}{2\pi}\int_{-\infty}^{\infty} \frac{\partial_H\log{\left(\hat{K}_H(\lambda)+1\right)}}{\alpha^2 + \lambda^2}\,d\lambda\\
\frac{\alpha}{2\pi}\int_{-\infty}^{\infty} \frac{\partial_H\log{\left(\hat{K}_H(\lambda)+1\right)}}{\alpha^2 + \lambda^2}\,d\lambda & \frac{1}{2\alpha} 
\end{pmatrix}.
\end{equation}
Here  $\hat{K}_H(\lambda) := \int_{-\infty}^{\infty}H(2H-1)\left\lvert \tau \right\rvert ^{2H-2}e^{-\mathrm{i}\lambda \tau} \,d\tau = \alpha_H \left\lvert \lambda \right\rvert ^{1-2H}$,$\lambda \in \mathbb{R}\backslash \{0\}$ with the constant
$\alpha_H := \Gamma(2H+1)sin(\pi H)$. The symbol $\partial_H$ denotes the partial derivative with respect to parameter variable $H$.
\end{thm}

\begin{remark}
The function $\hat{K}_H(\lambda)$ is Fourier transform of the kernel $K$ defined in \eqref{eq: K}. It does not decay suﬃciently fast to be integrable on $\mathbb{R}$ and hence, strictly speaking, it is not a spectral density of a stochastic process in the usual sense. So can we use this one to construct the multi-step MLE will be doubtable. However, theoretically the MLE can achieve the H\'aj\`ek lower bound.
\end{remark}

\subsection{Simulation Study}
As presented in \cite{CK23}, Whittle's formula in continuous time is a more subtle matter due to complexity of the absolute continuity relation between Gaussian measures on functions, it was never rigorously verified beyond processes with rational spectra. 

With the theorem \ref{th: LAN for OU}, the H\'aj\`ek asymptotic bound can be attained by the ML estimator, however the estimation procedure is not absolute. To find the estimation we will maximum the log-likelihood function
$$
\int_0^T b_t(X,\theta)\, dX_t 
- \frac{1}{2} \int_0^T b_t(X,\theta)^2 \, dt 
$$
with $b_t(X)$ defined below contains the the function $g(t,s;H)$, the  derivative with respect to  $H$ and the stochastic integral with respect to $X_t$. The realization needs a high precision numerical solution of the integral equation and approximation of the stochastic integral.

In our case, such a rate optimal estimator in the large time asymptotic regime can be constructed using the discrete samples of the observed continuous path $X_t,\, t\in [0,T]$ with a step $\Delta>0$. The discrete samples form a stationary sequence to which, e.g., Whittle's spectral estimator applied directly. In fact from the thesis of Hult \cite{Hult03} the spectral density of the discrete sample $\left(X_{1\Delta}, X_{2\Delta},\, \dots,\, X_{n\Delta}\right)$ will be 
$$
f_{\Delta}(\lambda)=\frac{a_H\Delta^{2H}}{2\pi}\sum_{k=\infty}^{\infty}\frac{|\lambda+2\pi k|^{1-2H}}{(\Delta \alpha)^2+(\lambda+2\pi k)^2}+ \frac{1}{2\alpha}\frac{1-e^{-2\alpha\Delta}}{1+e^{-2\alpha \Delta}-2\cos \lambda e^{-\alpha \Delta}}
$$
We take $\alpha=2,H=0.8$, the distance $\Delta=0.001$  and with step $N=100000$, we simulate 100 times and obtain the following result, the mean of of the two variable $\bar{\alpha}=2.13$ and $\bar{H}=0.85$. Because this paper not focus on  the problem of the simulation of discrete samples, so we will not present more details about this Whittle's estimation.  In fact, from the theory \cite{Dahlhaus89}, the limit risk can be made arbitrarily close to H\'aj\`ek with the Fisher information matrix \eqref{eq: fisher matrix} if $\Delta$ is chose small enough.

\section{Useful Lemmas for The Laplace Transform of Function $g(t,s)$}
In the next proof of the Theorem \ref{th: LAN for OU}, the properties of Laplace transform of the function $g(t,s)$ is the key point, following the work of Chiganky and Kleptsyna \cite{CK23}, we present here the useful Lemmas for it. First of all let us define define 
\begin{equation}\label{eq: Laplace of g exact}
\hat{g}_t(z;H)=\int_{-\infty}^{\infty}g(t,s;H)e^{-zs}ds=\int_0^t g(t,s)e^{-zs}ds,\, z\in \mathbb{C}, 
\end{equation}
Then we have 
\begin{lem}\label{lemma 1}
The Laplace transform $\hat{g}_t(z;H)$ can be presented in the following formula:
\begin{equation}
\label{eq: hat g moin 1}
\hat{g}_t(z;H) - 1 = \frac{\Phi_0(z) + e^{-tz} \Phi_1(-z)}{\Lambda(z)} 
\end{equation}
where
\begin{equation}
\label{eq: big Lambda}
\Lambda(z) = 1 + \frac{1}{2} \Gamma(2H + 1)  \left(z^{1 - 2H} + (-z)^{1 - 2H}  \right) 
\end{equation}
and the functions $ \Phi_0(z) $ and $ \Phi_1(z) $ are sectionally holomorphic on $ \mathbb{C} \backslash \mathbb{R}_+ $,
\begin{align}
\label{eq: phi1 and phi2}
\Phi_0(z) &= -1 + O(z^{1 - 2H}) \quad \text{and} \quad \Phi_1(z) = O(z^{1 - 2H}), \quad z \to \infty, \\
\Phi_0(z) &= O(z^{1 - 2H}) \quad \text{and} \quad \Phi_1(z) = O(z^{1 - 2H}), \quad z \to 0. 
\end{align}
\end{lem}

The next lemma gathers some useful properties $ \Lambda(z)$:
\begin{lem}\label{properties of big Lambda}
The function $ \Lambda(z) $ defined in \eqref{eq: big Lambda} is non-vanishing and sectionally holomorphic on $ \mathbb{C} \backslash \mathbb{R} $ with the limits
\begin{equation}\label{eq: pre big Lambda}
\Lambda^{\pm}(\tau) = 1 +  a_H |\tau|^{1 - 2H} \begin{cases}
e^{\pm (H - \frac{1}{2})\pi \mathrm{i}}, & \tau \in \mathbb{R}_+, \\
e^{\mp (H - \frac{1}{2})\pi \mathrm{i}}, & \tau \in \mathbb{R}_-,
\end{cases}
\end{equation}
where $ a_H = \Gamma(2H + 1) \sin(\pi H) $. For example, consider $\Lambda^{+}(\tau)$ with $\tau \in \mathbb{R}_+$, we have 
\begin{align*}
&z^{1-2H} \to \tau^{1-2H},\\
&(-z)^{1-2H} \to \tau^{1-2H}e^{-i\pi(1-2H)},   
\end{align*}
\begin{equation*}
\begin{aligned}
\Lambda^{+}(\tau) =& 1 + \frac{\Gamma(1+2H)}{2}|\tau|^{1-2H}\left(1+e^{-i\pi(1-2H)}\right)\\
=&1 + \frac{\Gamma(1+2H)}{2}|\tau|^{1-2H}\left(2\cos (\frac{1}{2}-H)\pi e^{-\mathrm{i}\pi(\frac{1}{2}-H)}\right)\\
=&1 + \Gamma(1+2H)|\tau|^{1-2H}\sin \left(H\pi\right)  e^{\mathrm{i}\pi(H-\frac{1}{2})}.         
\end{aligned}
\end{equation*}

These functions satisfy the symmetries
\begin{align}
\label{lt11}
\Lambda^{+}(\tau) &= \overline{\Lambda^{-}(\tau)},  \\
\frac{\Lambda^{+}(\tau)}{\Lambda^{-}(\tau)} &= \frac{\Lambda^{-}(-\tau)}{\Lambda^{+}(-\tau)}, 
\end{align}
and the principal branch of the argument $ \alpha(\tau) := \arg \{ \Lambda^{+}(\tau) \} $,
\begin{equation}
\label{eq:6.17}
\alpha(\tau) = \arctan \frac{ a_H \sin \big( (H - \frac{1}{2})\pi \big)}{|\tau|^{2H - 1} +  a_H \cos \big( (H - \frac{1}{2})\pi \big)} \operatorname{sign}(\tau), 
\end{equation}
is an odd decreasing function, continuous on $ \mathbb{R} \backslash \{0\} $, satisfying
\begin{equation}
\label{lt3}
\alpha(0+) = \pi (H - \tfrac{1}{2}) \quad \text{and} \quad \alpha(\tau) = O(\tau^{1 - 2H}) \quad \text{as } \tau \to \infty. 
\end{equation}
\end{lem}
Now with the Hilbert boundary value problem \cite{Gakhov} (details can be found in \cite{CK23}) we can get $\Lambda(z)$ with Sokhotski-Plemelj formula:
\begin{lem}
The function $\Lambda(z)$ can be identified by 
\begin{equation}
\label{eq: big Lambda equal product}
Y_{\mathrm{c}}(z) Y_{\mathrm{c}}(-z) = \Lambda(z), \quad z \in \mathbb{C} \backslash \mathbb{R}. 
\end{equation}
where $Y_{\mathrm{c}}(z)$  is defined by 
\begin{align}
\label{lt10}
Y_{\mathrm{c}}(z) &= \exp \left( \frac{1}{2\pi \mathrm{i}} \int_0^{\infty} \frac{\log \Lambda^{+}(\tau)/\Lambda^{-}(\tau)}{\tau - z} \, d\tau \right) \notag \\
&= \exp \left( \frac{1}{\pi} \int_0^{\infty} \frac{\alpha(\tau)}{\tau - z} \, d\tau \right), \quad z \in \mathbb{C} \backslash \mathbb{R}_+. 
\end{align}
It satisfies the asymptotic
\begin{equation}
\label{eq: asymptotic XCZ}
Y_{\mathrm{c}}(z) = \begin{cases}
O(z^{\frac{1}{2} - H}), & z \to 0, \\
1, & z \to \infty,
\end{cases} 
\end{equation}
\end{lem}
We define the complex function $Y(z)=z^kY_c(z)$ for some integer $k\in \mathbb{Z}$ and $h(\tau) := \frac{1}{2\mathrm{i}} \left( \frac{Y^{+}(\tau)}{Y^{-}(\tau)} - 1 \right) \frac{Y(-\tau)}{Y^{+}(\tau)}$, then 
\begin{lem}
The Laplace transform \eqref{eq: Laplace of g exact} satisfies
\begin{equation}\label{eq: laplace moin 1 exact}
\hat{g}_t(z;H) - 1 = -\frac{1}{Y(-z)} + \hat{R}_t(z), \quad z \in \mathbb{C}, 
\end{equation}
where
\begin{equation}
\label{ff3}
\hat{R}_t(z) := \frac{1}{Y(-z)} \big( p_t(-z) + q_t(-z) + 1 \big) + e^{-tz} \frac{1}{Y(z)} \big( p_t(z) - q_t(z) \big). 
\end{equation}
where 
\begin{equation}
\label{eq:6.27}
\begin{aligned}
p_t(s) &= \frac{1}{\pi} \int_0^{\infty} \frac{h(\tau) \mathrm{e}^{-t\tau}}{\tau + s} p_t(\tau) \, d\tau - \frac{1}{2},  \\
q_t(s) &= -\frac{1}{\pi} \int_0^{\infty} \frac{h(\tau) \mathrm{e}^{-t\tau}}{\tau + s} q_t(\tau) \, d\tau - \frac{1}{2}.    
\end{aligned}
\end{equation}
\end{lem}
The last lemma gives us the asymptotic analysis of $p_t(z)$ and $q_t(z)$
\begin{lem}
\label{lempq}
For any closed ball $B \subset \Theta$, there exist constants $r_{\max} > 0$, $T_{\min} > 0$ and $C > 0$ such that for any $r \in [0, r_{\max}]$ and all $t \geqslant T_{\min}$
\begin{equation*}
\int_{-\infty}^{\infty} |m_t(\mathrm{i}\lambda)|^2 |\lambda|^{-r}  d\lambda \leqslant C t^{r - 1},
\end{equation*}
where $m_t(z)$ is any of the functions
\begin{equation}
\Big\{ p_t(z) + \tfrac{1}{2},\ q_t(z) + \tfrac{1}{2},\ \partial_j p_t(z),\ \partial_j q_t(z),\ \partial_i \partial_j p_t(z),\ \partial_i \partial_j q_t(z) \Big\}.
\label{eq:6.32}
\end{equation}
\end{lem}

\section{Proof of Theorem \ref{th: LAN for OU}}
\subsection{The Proof Roadmap}
From formula \eqref{eq: BM} we have 
\begin{equation}
\label{eq: formula of X}
\begin{aligned}
X_t =& - \alpha \int_{0}^{t} X_s\,ds + \int_{0}^{t}\rho_s(M, H)\,ds + \overline{B}_t\\
=& \int_{0}^{t}b_s(X, \theta)\,ds + \overline{B}_t,
\end{aligned}
\end{equation}
where $b_t(X, \theta):= - \alpha X_t + \int_{0}^{t}g(t, t-s;H)\,dX_s + \alpha \int_{0}^{t}g(t, t-s;H)X_s\,ds$, recall that $\theta=(H,\alpha)$.

Let $\mathbf{P}^T$ and $\mathbf{P}^T_\theta$ be the probability measures on  $C([0,T],\mathbb{R})$ induced by the Brownian motion $\overline{B}$ and the mixed fractional  O-U process with parameter $\theta$, respectively. 
By the Girsanov theorem, applied to the innovation representation eq.\eqref{eq: formula of X},  $\mathbf{P}^T \sim \mathbf{P}^T_\theta$ with the Radon-Nikodym derivative
\begin{equation} \label{eq: Girsanov initial}
\frac{d\mathbf{P}^T_\theta}{d\mathbf{P}^T}(X^T) 
= \exp\left( 
\int_0^T b_t(X,\theta)\, dX_t 
- \frac{1}{2} \int_0^T b_t(X,\theta)^2 \, dt 
\right).  
\end{equation}
\begin{remark}
One can easily check $\mathbf{E}\left[\int_{0}^{T}b_s^2\,ds\right] < \infty$.
\end{remark}
From \eqref{eq: Girsanov initial}, we have 
\begin{equation}
\label{pf4}
\begin{aligned}
\frac{d\mathbf{P}_{\theta_0+\phi(T)u}}{d\mathbf{P}_{\theta_0}}(X^T) =& \exp(\int_{0}^{T}b_t(X,\theta_0+\phi(T)u)-b_t(X,\theta_0)\,d\overline{B}_t\\
&-\frac{1}{2}\int_{0}^{T}(b_t(X,\theta_0+\phi(T)u)-b_t(X,\theta_0))^2\,dt), 
\end{aligned}
\end{equation}
We can finish the proof of the LAN property in Theorem \ref{th: LAN for OU} if it is true for
\begin{equation}
\label{pf37}
\int_{0}^{T}(b_t(X,\theta_0+u/\sqrt{T})-b_t(X,\theta_0))^2\,dt \xrightarrow[\text{$T \to \infty$}]{\mathbf{P}} u^{\top} I(\theta_0) u,
\end{equation}
if \eqref{pf37} holds, then by the CLT for stochastic integrals \cite{kutoyants}, it implies that the following convergence in distribution
\begin{equation}
\int_{0}^{T}\left(b_t(X,\theta_0+u/\sqrt{T})-b_t(X,\theta_0)\right)\,d\overline{B}_t \xrightarrow[\text{$T \to \infty$}]{\mathrm{d}} u^{\top} N(0,I(\theta_0)^{1/2}).
\end{equation}
Let us define $\partial_1 := \partial_H$ and $\partial_2 := \partial_{\alpha}$. The kernel $K_H$ defined in \eqref{eq: K} is $L^2$- integrable on the interval $[0,T]$ for $\tau \neq 0$ and 
$$
dK_H/dH(\cdot),\, d^2 K_H/dH^2(\cdot)\in L^2([0,T]), 
$$ 
then the solution to the equation \eqref{eq: g} also has
$$
dg(T, \cdot: H)/dH,\, d^2g(T,\cdot: H)/dH^2\in L^2([0,T]),
$$
which can be interchanged with the stochastic integral in \eqref{eq: rho}. Consequently 
\begin{align}
&\partial_1 b_t(X,\theta) = \int_{0}^{t}\partial_1 g(t,t-s;H)\,dX_s + \alpha \int_{0}^{t}X_s \partial_1 g(t,t-s;H)\,ds,\\
&\partial_2 b_t(X,\theta) = -X_t + \int_{0}^{t}X_s g(t,t-s;H)\,ds.
\end{align}
Let $\nabla$ and $\nabla^2$ denote the gradient and the Hessian operators with respect to $\theta$, using the integral mean value theorem we have
\begin{equation}
\begin{aligned}
&b_t(X,\theta_0+u/\sqrt{T})-b_t(X,\theta_0)= \\
&\frac{1}{\sqrt{T}}\nabla b_t(X,\theta_0)u 
+ \frac{1}{T}\int_{0}^{1} \int_{0}^{\tau}  u^\top \nabla^2 b_t\left(X, \theta_0+su/\sqrt{T}\right)u\,ds  \,d\tau.
\end{aligned}
\end{equation}
To prove \eqref{pf37}, we only need to ensure the two limits
\begin{equation}
\label{pf5}
\frac{1}{T}\int_{0}^{T}\nabla^\top b_t(X,\theta_0) \nabla b_t(X,\theta_0)\,dt \xrightarrow[\text{$T \to \infty$}]{L^2(\Omega)} I(\theta_0),
\end{equation}
and
\begin{equation}
\label{pf38}
\frac{1}{T^2}\int_{0}^{T} \sup_{\left\lVert \theta - \theta_0\right\rVert \le \delta }  \mathbb{E}\left\lVert \nabla^2 b_t(X,\theta)\right\rVert ^2  \,dt \xrightarrow[\text{$T \to \infty$}]{} 0, 
\end{equation}
are both true. When $X_0=0$, we can easily find the solution for the model \eqref{eq: mixed OU}
\begin{equation}\label{eq: solution of O-U}
X_t = \int_{0}^{t} e^{-\alpha(t-u)}\,dM_u
\end{equation}

The key lemmas we need to prove \eqref{pf5} is as following,
\begin{lem}
\label{key}
Let us define
$$
J(s, t;\theta):= \nabla^\top b_s\left(X,\theta\right) \nabla b_t\left(X,\theta\right),
$$
then with the true value $\theta_0$ the expectation $J(s, t;\theta_0)$ satisfies 
\begin{equation}
\mathbf{E}J(s, t;\theta_0) = Q(s,t) + R(s,t),
\end{equation}
where for $t$, $s$ sufficiently large, 
\begin{equation}
\begin{aligned}
&\left\lvert Q(s,t) \right\rvert \le C \wedge \left\lvert t-s \right\rvert^{-1}\left\lvert \log|t-s|\right\rvert^3, \forall s,t \in \mathbb{R}^+,\\
&\left\lvert R(s,t) \right\rvert \le C\left(t^{-1/2} + s^{-1/2} + (st)^{-b}\right),\forall s,t \in [T_{min},\infty),
\end{aligned}
\end{equation}
with some constants $b \in \left(0, \frac{1}{2}\right) $, $C > 0$ and $T_{min} > 0$. And 
\begin{equation}
Q(t,t) \to
\begin{pmatrix}
\frac{1}{4\pi}\int_{-\infty}^{\infty}(\partial_H \log(\hat{K}_{H_0}(\lambda)+1))^2\,d\lambda& 
\frac{\alpha_0}{2\pi}\int_{-\infty}^{\infty} \frac{\partial_H\log{\left(\hat{K}_{H_0}(\lambda)+1\right)}}{\alpha_0^2 + \lambda^2}\,d\lambda\\
\frac{\alpha_0}{2\pi}\int_{-\infty}^{\infty} \frac{\partial_H\log{\left(\hat{K}_{H_0}(\lambda)+1\right)}}{\alpha_0^2 + \lambda^2}\,d\lambda & 
\frac{1}{2\alpha_0} 
\end{pmatrix}
\end{equation}
as $t \to \infty$.
\end{lem}

The next lemma we need is 
\begin{lem}
\label{key2}
for all sufficiently small $\delta > 0$ there exist constants $C > 0$ and $T_{min} > 0$ such that
\begin{equation}
\sup_{\left\lVert \theta - \theta_0\right\rVert \le \delta }  \mathbf{E}\left\lVert \nabla^2 b_t(X, \theta)\right\rVert ^2 \le C, \quad \forall t \ge T_{min},
\end{equation}
\end{lem}

\begin{remark}
In fact the lemma \ref{key} implies 
\begin{equation}
\label{pf35}
\mathbf{E}J(t,t;\theta_0) \to
\begin{pmatrix}
\frac{1}{4\pi}\int_{-\infty}^{\infty}(\partial_H \log(\hat{K}_{H_0}(\lambda)+1))^2\,d\lambda& 
\frac{\alpha_0}{2\pi}\int_{-\infty}^{\infty} \frac{\partial_H\log{\left(\hat{K}_{H_0}(\lambda)+1\right)}}{\alpha_0^2 + \lambda^2}\,d\lambda\\
\frac{\alpha_0}{2\pi}\int_{-\infty}^{\infty} \frac{\partial_H\log{\left(\hat{K}_{H_0}(\lambda)+1\right)}}{\alpha_0^2 + \lambda^2}\,d\lambda & 
\frac{1}{2\alpha_0}  
\end{pmatrix}
\end{equation} 
as $t \to \infty$
and 
\begin{equation}
\label{pf11}
\frac{1}{T^2}\int_{0}^{T} \int_{0}^{T} \left\lVert \mathbb{E}J(s,t;\theta_0) \right\rVert^2  \,ds \,dt \xrightarrow[\text{$T \to \infty$}]{} 0.    
\end{equation}
By Lemma \ref{key2}, the limit \eqref{pf38} holds. It can be easily verified that $\nabla b_t\left(X,\theta_0\right)$ is a centered Gaussian process. By \eqref{pf11} and \eqref{pf35} these limits ensure \eqref{pf5} is true.
\end{remark} The next two subsections will contribute to the proofs of Lemma \ref{key} and Lemma \ref{key2}.

\subsection{Proof of Lemma \ref{key}}
Consider the partial derivatives, in view of eq.\eqref{eq: solution of O-U}, we have 
$$
\partial_1 b_s = \int_{0}^{s}\partial_1 g(s, s-r; H)\,\left[dX_r + \alpha X_r dr\right]=  \int_{0}^{s}\partial_1 g(s, s-r; H)\,dM_r,         
$$
$$
\partial_2 b_s= -\int_{0}^{t}e^{-\alpha(t-s)}\,dM_s + \int_{0}^{t}\int_{r}^{t}g(t,t-s ;H)e^{-\alpha(s-r)}\,dsdM_r.         
$$
Let $l(t,u)=-e^{-\alpha u}\mathbf{1}_{[0,t]}$, then 
\begin{equation*}
\begin{aligned}
h(t, r) :=& \int_{r}^{t}g(t,t-s;H)e^{-\alpha(s-r)}\,ds\\
=&\int_{0}^{t-r}g(t,s;H)e^{-\alpha(t-s-r)}\,ds\\
=&(g(t, \cdot)*(-l(t,\, \cdot)))(t-\tau) 
\end{aligned}
\end{equation*}
where $(g(t, \cdot)*(-l(t,\, \cdot)))(s) = - \int_{0}^{s}g(t,\tau)l(t,s-\tau)\,d\tau$ is convolution product for $t$ fixed. 

Now when  the Laplace transform of  $l(t,u)$ is 
$$
\hat{l}_t(z) = \int_{0}^{t}e^{-zs}l(s)\,ds=\frac{1}{z+\alpha} \left(e^{-(z+\alpha)t}-1\right),
$$
with the convolution theorem we have the Laplace Transform of $h(t,\, \cdot)$ can be presented by 
$$
\hat{h}_t(z) = -\hat{g}_t(z)\frac{1}{z+\alpha}\left(e^{-(z+\alpha)t}-1\right)
$$

From now on, in this subsection we will omit $H_0$ from the notations and denote the true value $\alpha_0$ by $\alpha$ for brevity.
Firstly, we consider 
\begin{equation}
\begin{aligned}
&\mathbf{E}\partial_1 b_s(X) \partial_1 b_t(X)\\ 
=&\mathbf{E}\int_{0}^{s}\partial_1 g(s,s-x)[dX_x + \alpha X_x\,dx]\int_{0}^{t}\partial_1 g(t,t-y)[dX_y + \alpha X_y\,dy] \\
=&\mathbf{E}\int_{0}^{s}\partial_1 g(s,s-x)\,dM_x\int_{0}^{t}\partial_1 g(t,t-y)\,dM_y\\
=&\mathbf{E}\partial_1 \rho_s(X) \partial_1 \rho_t(X).
\end{aligned}
\end{equation}
From the \cite{CK18} we have 
\begin{equation}
\begin{aligned}
&\mathbf{E}\partial_1 b_s(X) \partial_1 b_t(X)\\ 
=&Q_{11}(s,t) + R^{(1)}_{11}(s, t) + R^{(2)}_{11}(s, t) + R^{(3)}_{11}(s, t),    
\end{aligned}
\end{equation}
where 
\begin{equation*}
Q_{11}(s,t) := \frac{1}{2\pi} \int_{-\infty}^{\infty} \partial_1 \frac{1}{Y(-\mathrm{i}\lambda)} 
\overline{\partial_1 \frac{1}{Y(\mathrm{i}\lambda)}} \Lambda(\mathrm{i}\lambda) e^{\mathrm{i}(t - s)\lambda} \, d\lambda 
\end{equation*}
and
\begin{align*}
R^{(1)}_{11}(s, t) &:= -\frac{1}{2\pi} \int_{-\infty}^{\infty} \partial_1 \frac{1}{Y(-\mathrm{i}\lambda)} 
\overline{\partial_1 \hat{R}_s(\mathrm{i}\lambda)} \Lambda(\mathrm{i}\lambda) e^{\mathrm{i}(t - s)\lambda} \, d\lambda,  \\
R^{(2)}_{11}(s, t) &:= -\frac{1}{2\pi} \int_{-\infty}^{\infty} \partial_1 \hat{R}_t(\mathrm{i}\lambda)
 \overline{\partial_1 \frac{1}{Y(-\mathrm{i}\lambda)}} \Lambda(\mathrm{i}\lambda) e^{\mathrm{i}(t - s)\lambda} \, d\lambda,  \\
R^{(3)}_{11}(s, t) &:= \frac{1}{2\pi} \int_{-\infty}^{\infty} \partial_1 \hat{R}_t(\mathrm{i}\lambda) 
\overline{\partial_1 \hat{R}_s(\mathrm{i}\lambda)} \Lambda(\mathrm{i}\lambda) e^{\mathrm{i}(t - s)\lambda} \, d\lambda.
\end{align*}
Next we consider 
\begin{equation}
\begin{aligned}
&\mathbf{E}\partial_1 b_s(X) \partial_2 b_t(X)\\
=&\mathbf{E}\int_{0}^{s}\partial_1 g(s,s-r)\,dM_r \left[\int_{0}^{t}l(t,t-s)\,dM_s + h(t,r)dM_r\right] \\
=&\int_{0}^{s}\int_{0}^{t}\partial_1 g(s,s-r)\left[l(t,t-u) + h(t,u)\right]K_H(r-u)\,dudr\\
&+\int_{0}^{s}\partial_1 g(s,s-r)\left[l(t,t-r) + h(t,r)\right]\,dr\\
=&\int_{0}^{s}\int_{0}^{t}\partial_1 g(s,r)\left[l(t,u) + h(t,t-u)\right]K_H(s-t-r+u)\,dudr\\
&+\int_{0}^{s}\partial_1 g(s,r)\left[l(t, t-s+r) + h(t,s-r)\right]\,dr\\
\end{aligned}
\end{equation}
By the classical Plancherel's theorem, we have 
\begin{equation}
\begin{aligned}
&\mathbf{E}\partial_1 b_s(X) \partial_2 b_t(X)\\
=&\frac{1}{2\pi}\int_{-\infty}^{\infty}\partial_1 \hat{g}_s(\mathrm{i}\lambda)
\overline{\hat{l}_t(\mathrm{i}\lambda)+\hat{h}_t(\mathrm{i}\lambda)}\Lambda(\mathrm{i}\lambda) e^{\mathrm{i}(t - s)\lambda}\,d\lambda, 
\end{aligned}
\end{equation}
and in view of the Laplace Transform of $l(t,u)$ and $h(t,u)$, we have 
\begin{equation}
\label{pf10}
\begin{aligned}
&\mathbb{E}\partial_1 b_s(X) \partial_2 b_t(X)\\
=&\frac{1}{2\pi}\int_{-\infty}^{\infty}\partial_1 \hat{g}_s(\mathrm{i}\lambda)
\left[1 - \overline{\hat{g}_t(\mathrm{i}\lambda)}\right] \frac{1}{\alpha-\mathrm{i}\lambda}
\left(e^{\left(\mathrm{i}\lambda - \alpha\right)t}- 1\right) \Lambda(\mathrm{i}\lambda) e^{\mathrm{i}(t - s)\lambda}\,d\lambda\\
=&\frac{1}{2\pi}\int_{-\infty}^{\infty}\left(-\partial_1\frac{1}{Y(-\mathrm{i}\lambda)} + \partial_1 \hat{R}_s(\mathrm{i}\lambda)\right) 
\left(\frac{1}{Y(\mathrm{i}\lambda)} - \hat{R}_t(-\mathrm{i}\lambda)\right) \frac{1}{\alpha-\mathrm{i}\lambda}
\left(e^{\left(\mathrm{i}\lambda - \alpha\right)t}- 1\right) \Lambda(\mathrm{i}\lambda) e^{\mathrm{i}(t - s)\lambda}\,d\lambda\\
=&Q_{12}(s, t) + R^{(1)}_{12}(s, t) + R^{(2)}_{12}(s, t) + R^{(3)}_{12}(s, t), 
\end{aligned} 
\end{equation}
where 
\begin{equation*}
Q_{12}(s, t) := -\frac{1}{2\pi}\int_{-\infty}^{\infty}\partial_1\frac{1}{Y(-\mathrm{i}\lambda)}  
\frac{1}{Y(\mathrm{i}\lambda)} \frac{1}{\alpha-\mathrm{i}\lambda}
\left(e^{\left(\mathrm{i}\lambda - \alpha\right)t}- 1\right) \Lambda(\mathrm{i}\lambda) e^{\mathrm{i}(t - s)\lambda}\,d\lambda\\
\end{equation*}
and
\begin{align*}
R^{(1)}_{12}(s, t) &:= \frac{1}{2\pi}\int_{-\infty}^{\infty}\partial_1 \hat{R}_s(\mathrm{i}\lambda)  
\frac{1}{Y(\mathrm{i}\lambda)} \frac{1}{\alpha-\mathrm{i}\lambda}
\left(e^{\left(\mathrm{i}\lambda - \alpha\right)t}- 1\right) \Lambda(\mathrm{i}\lambda) e^{\mathrm{i}(t - s)\lambda}\,d\lambda,\\
R^{(2)}_{12}(s, t) &:= \frac{1}{2\pi}\int_{-\infty}^{\infty} \hat{R}_t(-\mathrm{i}\lambda)  
\partial_1\frac{1}{Y(-\mathrm{i}\lambda)} \frac{1}{\alpha-\mathrm{i}\lambda}
\left(e^{\left(\mathrm{i}\lambda - \alpha\right)t}- 1\right) \Lambda(\mathrm{i}\lambda) e^{\mathrm{i}(t - s)\lambda}\,d\lambda,\\
R^{(3)}_{12}(s, t) &:= -\frac{1}{2\pi} \int_{-\infty}^{\infty} \partial_1 \hat{R}_s(\mathrm{i}\lambda) 
\hat{R}_t(-\mathrm{i}\lambda) \Lambda(\mathrm{i}\lambda) e^{\mathrm{i}(t - s)\lambda} \, d\lambda.
\end{align*}

From \eqref{pf10}, we can instantly give
\begin{equation}
\begin{aligned}
&\mathbf{E}\partial_2 b_s(X) \partial_1 b_t(X)\\
=&Q_{21}(s, t) + R^{(1)}_{21}(s, t) + R^{(2)}_{21}(s, t) + R^{(3)}_{21}(s, t), 
\end{aligned}   
\end{equation}
where
\begin{align}
&Q_{21}(s, t) = Q_{12}(t, s),\\
&R^{(i)}_{21}(s, t) = R^{(i)}_{12}(t, s),\quad i=1,2,3.
\end{align}
From the same way, we can conclude the $(2,2)$-component:
\begin{equation*}
\begin{aligned}  
&\mathbf{E}\left[\partial_2 b_s(X) \partial_2 b_t(X)\right]\\
=&\frac{1}{2\pi}\int_{-\infty}^{\infty}\left[\hat{l}_s(\mathrm{i}\lambda)\overline{\hat{l}_t(\mathrm{i}\lambda)}+\hat{l}_s(\mathrm{i}\lambda)\overline{\hat{h}_t(\mathrm{i}\lambda)}+
\hat{h}_s(\mathrm{i}\lambda)\overline{\hat{l}_t(\mathrm{i}\lambda)}+\hat{h}_s(\mathrm{i}\lambda)\overline{\hat{h}_t(\mathrm{i}\lambda)}\right]
\Lambda(\mathrm{i}\lambda) e^{\mathrm{i}(t - s)\lambda}\,d\lambda\\ 
=&\frac{1}{2\pi}\int_{-\infty}^{\infty}\left(1+\hat{g}_s(\mathrm{i}\lambda)
\overline{\hat{g}_t(\mathrm{i}\lambda)}-\overline{\hat{g}_t(\mathrm{i}\lambda)}-\hat{g}_s(\mathrm{i}\lambda)\right)
\frac{1}{\alpha^2+\lambda^2}\left(e^{-(\alpha+\mathrm{i}\lambda)s}-1\right) \left(e^{-(\alpha-\mathrm{i}\lambda)t}-1\right) \Lambda(\mathrm{i}\lambda) e^{\mathrm{i}(t - s)\lambda}\,d\lambda\\ 
=&Q_{22}(s, t) + R^{(1)}_{22}(s, t) + R^{(2)}_{22}(s, t) + R^{(3)}_{22}(s, t),
\end{aligned}   
\end{equation*}
where
\begin{equation*}
Q_{22}(s, t)=\frac{1}{2\pi}\int_{-\infty}^{\infty} \frac{1}{Y(-\mathrm{i}\lambda)}\frac{1}{Y(\mathrm{i}\lambda)} 
\frac{1}{\alpha^2 + \lambda^2}(e^{-(\alpha + \mathrm{i}\lambda) s}-1)(e^{(\mathrm{i}\lambda -\alpha) t}-1) \Lambda(\mathrm{i}\lambda)e^{i(t-s)\lambda}\,d\lambda
\end{equation*}
and
\begin{align*}
R^{(1)}_{22}(s, t) &:= -\frac{1}{2\pi}\int_{-\infty}^{\infty} \hat{R}_s(\mathrm{i}\lambda)  
\frac{1}{Y(\mathrm{i}\lambda)}\frac{1}{\alpha^2 + \lambda^2}(e^{-(\alpha + \mathrm{i}\lambda) s}-1)(e^{(\mathrm{i}\lambda -\alpha) t}-1) \Lambda(\mathrm{i}\lambda) e^{\mathrm{i}(t - s)\lambda}\,d\lambda,\\
R^{(2)}_{22}(s, t) &:= -\frac{1}{2\pi}\int_{-\infty}^{\infty} \hat{R}_t(-\mathrm{i}\lambda)  
\frac{1}{Y(-\mathrm{i}\lambda)} \frac{1}{\alpha^2 + \lambda^2}(e^{-(\alpha + \mathrm{i}\lambda) s}-1)(e^{(\mathrm{i}\lambda -\alpha) t}-1)
 \Lambda(\mathrm{i}\lambda) e^{\mathrm{i}(t - s)\lambda}\,d\lambda,\\
R^{(3)}_{22}(s, t) &:= \frac{1}{2\pi} \int_{-\infty}^{\infty} \hat{R}_s(\mathrm{i}\lambda) 
\hat{R}_t(-\mathrm{i}\lambda) \frac{1}{\alpha^2 + \lambda^2}(e^{-(\alpha + \mathrm{i}\lambda) s}-1)(e^{(\mathrm{i}\lambda -\alpha) t}-1) \Lambda(\mathrm{i}\lambda) e^{\mathrm{i}(t - s)\lambda} \, d\lambda.
\end{align*}

For $Q_{ij}(s,t)$, it can be proved that
\begin{lem}
\label{lemQ}
There exists $C > 0$ such that
\begin{equation*}
\begin{aligned}
\left\lvert Q_{11}(s,t)\right\rvert &\le C \wedge \left\lvert t-s \right\rvert^{-1}\left\lvert \log|t-s|\right\rvert^3, \\
\left\lvert Q_{ij}(s,t)\right\rvert &\le C \wedge \left\lvert t-s \right\rvert^{-1}\left\lvert \log|t-s|\right\rvert^2, \quad for \quad i=1,\, j=2  \quad or \quad i=2,\, j=1\\
\left\lvert Q_{22}(s,t)\right\rvert &\le C e^{-\alpha |t-s|}. \\
\end{aligned}
\end{equation*}
Moreover, 
\begin{equation*}
Q_{11}(t,t) = \frac{1}{4\pi}\int_{-\infty}^{\infty}\left(\partial_1\log{\left(\hat{K}_{H}(\lambda)+1\right) }\right)^2 \,d\lambda,\\
\end{equation*}
and as $t \to \infty$,
\begin{align*}
&Q_{12}(t,t) = Q_{21}(t,t) \to \frac{1}{2\pi}\int_{-\infty}^{\infty} \frac{\alpha}{\alpha^2 + \lambda^2}\partial_1\log{\left(\hat{K}_{H}(\lambda)+1\right) }\,d\lambda,\\
&Q_{22}(t,t) \to \frac{1}{2\alpha}.
\end{align*}
\end{lem}

\begin{proof}
Firstly, as shown in \cite{CK23}, instantly we have 
\begin{equation*}
|Q_{11}(s, t)| \le C \wedge \left\lvert t-s \right\rvert^{-1}\left\lvert \log|t-s|\right\rvert^3.      
\end{equation*}
Considering the equation  \eqref{eq: big Lambda equal product}, let us define
\begin{equation}
\begin{aligned}
f(\lambda) :=& \partial_1\frac{1}{Y(-\mathrm{i}\lambda)}  
\frac{1}{Y(\mathrm{i}\lambda)}\Lambda(\mathrm{i}\lambda)\\
=&\partial_1 \log Y(-\mathrm{i}\lambda),
\end{aligned}
\end{equation}
from the equation (6.36) of \cite{CK23}, we have 
\begin{equation}
\label{ff2}
f(\lambda) = 
\begin{aligned}
\begin{cases}
O(\log|\lambda|^{-1}),&\lambda \to 0,\\
O(|\lambda|^{1-2H_0}\log|\lambda|),&\lambda \to \pm \infty
\end{cases}  
\end{aligned}
\end{equation}
and it can be easily to see that $f\in L^1(\mathbb{R})$. It follows that
\begin{equation}
\left\lvert Q_{12}(s,t)\right\rvert \le \left\lVert f \right\rVert _1.     
\end{equation}
and in view of \eqref{eq: asymptotic XCZ},
\begin{equation}
\label{ff1}
\partial_1 \frac{1}{Y(\mathrm{i}\lambda)} = -\frac{\partial_1 \log Y(\mathrm{i}\lambda)}{Y(\mathrm{i}\lambda)}=
\begin{aligned}
\begin{cases}
O(|\lambda|^{H_0-1/2}\log^2|\lambda|^{-1}),&\lambda \to 0,\\
O(|\lambda|^{1-2H_0}\log|\lambda|),&\lambda \to \pm \infty,
\end{cases}  
\end{aligned}
\end{equation}
Also in view of \eqref{ff1}, \eqref{eq: asymptotic XCZ} and \eqref{eq: big Lambda}, we have  
\begin{equation}
f^{\prime}(\lambda) = 
\begin{aligned}
\begin{cases}
O(|\lambda|^{-1}\log|\lambda|^{-1}),&\lambda \to 0,\\
O(|\lambda|^{-2H_0}\log|\lambda|),&\lambda \to \pm \infty.
\end{cases}  
\end{aligned}
\end{equation}
Consider
\begin{equation}
\begin{aligned}
& \int_{-\infty}^{\infty}f(\lambda)e^{i(t-s)\lambda}\,d\lambda\\
=& \int_{|\lambda| < |t-s|^{-1}}^{}f(\lambda)e^{i(t-s)\lambda}\,d\lambda + \int_{|\lambda| \ge |t-s|^{-1}}^{}f(\lambda)e^{i(t-s)\lambda}\,d\lambda\\
=&\mathrm{I} + \mathrm{II},
\end{aligned}    
\end{equation}
\begin{equation}
\begin{aligned}
& \left\lvert \mathrm{I}\right\rvert \\
\le& \int_{|\lambda| < |t-s|^{-1}}^{}\left\lvert f(\lambda)\right\rvert \,d\lambda \\
\le& C_1 \int_{|\lambda| < |t-s|^{-1}}^{}\log|\lambda|^{-1} \,d\lambda \\
\le& C_2 \left\lvert t-s \right\rvert^{-1}\left\lvert \log|t-s|\right\rvert^2,\\
\end{aligned}    
\end{equation}
and
\begin{equation}
\begin{aligned}
& \left\lvert \mathrm{II^{+}}\right\rvert \\
\le& \left\lvert f(\lambda)\left.\frac{e^{i(t-s)\lambda}}{i(t-s)}\right|^{\infty}_{|t-s|^{-1}}-
\int_{|t-s|^{-1}}^{\infty} \frac{e^{i(t-s)\lambda}}{i(t-s)} f^\prime(\lambda) \,d\lambda \right\rvert\\
\le& \frac{f(|t-s|^{-1})}{\left\lvert t-s\right\rvert } + \left\lvert t-s\right\rvert ^{-1} 
\left(\int_{|t-s|^{-1}}^{1} \left\lvert f^\prime(\lambda)\right\rvert \,d\lambda + C\right)  \\
\le& C_3 \left(\left\lvert t-s \right\rvert^{-1}\left\lvert \log|t-s|\right\rvert + \left\lvert t-s \right\rvert^{-1}\left\lvert \log|t-s|\right\rvert^2 + \left\lvert t-s\right\rvert ^{-1}
\right) ,\\
\le& C_4 \left\lvert t-s \right\rvert^{-1}\left\lvert \log|t-s|\right\rvert^2.
\end{aligned}    
\end{equation} 
To deal with $\left\lvert \mathrm{II^{-}}\right\rvert$ is something similar, so finally we have 
\begin{equation}
|Q_{12}(s, t)| \le \left\lvert \int_{-\infty}^{\infty}f(\lambda)e^{i(t-s)\lambda}\,d\lambda\right\rvert \le C\left\lvert t-s \right\rvert^{-1}\left\lvert \log|t-s|\right\rvert^2.
\end{equation}
the decay of $Q_{21}(s,t)$ will be the same of $Q_{12}(s,t)$.

Then, in view of \eqref{eq: big Lambda equal product},
\begin{equation*}
\begin{aligned}
Q_{22}(s, t) =&\frac{1}{2\pi}\int_{-\infty}^{\infty} \frac{1}{Y(-\mathrm{i}\lambda)}\frac{1}{Y(\mathrm{i}\lambda)} 
\frac{1}{\alpha^2 + \lambda^2}(e^{-(\alpha + \mathrm{i}\lambda) s}-1)(e^{(\mathrm{i}\lambda -\alpha) t}-1) \Lambda(\mathrm{i}\lambda)e^{i(t-s)\lambda}\,d\lambda\\
=&\frac{1}{2\pi}\int_{-\infty}^{\infty} 
\frac{1}{\alpha^2 + \lambda^2}(e^{-(\alpha + \mathrm{i}\lambda) s}-1)(e^{(\mathrm{i}\lambda -\alpha) t}-1)e^{i(t-s)\lambda}\,d\lambda.
\end{aligned}
\end{equation*}
Instantly, we have 
\begin{equation*}
|Q_{22}(s, t)| \le \frac{C}{2\alpha} e^{-\alpha |t-s|}.
\end{equation*}

Now let  $s=t$, first of all 
\begin{equation}
Q_{11}(t,t) = \frac{1}{4\pi}\int_{-\infty}^{\infty} \partial_1\log\left(1+\hat{K}_H(\lambda)\right) \partial_1\log\left(1+\hat{K}_H(\lambda)\right)\,d\lambda.
\end{equation}
which has been proved in \cite{CK23}. Then 
$$
Q_{12}(t,t)
=\frac{1}{2\pi}\int_{-\infty}^{\infty} \frac{1}{\alpha - \mathrm{i}\lambda} \partial_1\log{Y(-\mathrm{i}\lambda)} \,d\lambda- \frac{e^{-\alpha t}}{2\pi}\int_{-\infty}^{\infty} \frac{1}{\alpha - \mathrm{i}\lambda} \partial_1\log{Y(-\mathrm{i}\lambda)}e^{\mathrm{i}\lambda t} \,d\lambda,
$$
denote $g(\lambda) = -\frac{1}{2\pi}\int_{-\infty}^{\infty} \frac{1}{\alpha - \mathrm{i}\lambda} \partial_1\log{Y(-\mathrm{i}\lambda)} \,d\lambda$,
Since $\Lambda(z)$ is an even function, we have 
$$
\int_{-\infty}^{\infty} \frac{1}{\alpha - \mathrm{i}\lambda} \partial_1\log{\Lambda(\mathrm{i}\lambda)} \,d\lambda
=\int_{-\infty}^{\infty} \frac{1}{\alpha + \mathrm{i}\lambda} \partial_1\log{\Lambda(\mathrm{i}\lambda)} \,d\lambda
$$
and moreover, 
$$
\begin{aligned}
&\int_{-\infty}^{\infty} \frac{1}{\alpha - \mathrm{i}\lambda} \partial_1\log{\Lambda(\mathrm{i}\lambda)} \,d\lambda\\
=&\frac{1}{2}\left(\int_{-\infty}^{\infty} \frac{1}{\alpha - \mathrm{i}\lambda} \partial_1\log{\Lambda(\mathrm{i}\lambda)} \,d\lambda 
+ \int_{-\infty}^{\infty} \frac{1}{\alpha + \mathrm{i}\lambda} \partial_1\log{\Lambda(\mathrm{i}\lambda)} \,d\lambda\right)\\
=&\int_{-\infty}^{\infty} \frac{\alpha}{\alpha^2 + \lambda^2} \partial_1\log{\Lambda(\mathrm{i}\lambda)} \,d\lambda.
\end{aligned}
$$
From the previous conclude $\Lambda(-i\lambda)=Y(i\lambda)Y(i\lambda)$ we have  
$$
\begin{aligned}
&\int_{-\infty}^{\infty} \frac{1}{\alpha - \mathrm{i}\lambda} \partial_1\log{\Lambda(\mathrm{i}\lambda)} \,d\lambda \\
=&\int_{-\infty}^{\infty} \frac{1}{\alpha - \mathrm{i}\lambda} \partial_1\log{Y(-\mathrm{i}\lambda)} \,d\lambda + 
\int_{-\infty}^{\infty} \frac{1}{\alpha - \mathrm{i}\lambda} \partial_1\log{Y(\mathrm{i}\lambda)} \,d\lambda\\
=&\int_{-\infty}^{\infty} \frac{1}{\alpha - \mathrm{i}\lambda} \partial_1\log{Y(-\mathrm{i}\lambda)} \,d\lambda + 
\frac{1}{\pi}\int_{0}^{\infty}\partial_1\alpha(\tau)\left(\int_{-\infty}^{\infty} \frac{1}{\alpha - \mathrm{i}\lambda} 
\frac{1}{\tau-\mathrm{i}\lambda} \,d\lambda \right) d\tau,\\
\end{aligned}
$$
where $\alpha(\tau)=\arg\{\Lambda^+{\tau}\}$ which has been presented before and the last equality is because of \eqref{lt10}. Notice that 
$$
\int_{-\infty}^{\infty} \frac{1}{\alpha - \mathrm{i}\lambda} 
\frac{1}{\tau-\mathrm{i}\lambda} \,d\lambda = 0    
$$
the function $g(\lambda)$ will be 
$$
g(\lambda) = \int_{-\infty}^{\infty} \frac{\alpha}{\alpha^2 + \lambda^2} \partial_1\log{\Lambda(\mathrm{i}\lambda)} \,d\lambda.
$$
Now we return to the function $Q_{12}(t,t)$
$$
Q_{12}(t,t)
=\frac{1}{2\pi}g(\lambda)- \frac{e^{-\alpha t}}{2\pi}\int_{-\infty}^{\infty} \frac{1}{\alpha - \mathrm{i}\lambda} \partial_1\log{X(-\mathrm{i}\lambda)}e^{\mathrm{i}\lambda t} \,d\lambda,
$$
when we have the control
$$
\left\lvert \int_{-\infty}^{\infty} \frac{1}{\alpha - \mathrm{i}\lambda} \partial_1\log{X(-\mathrm{i}\lambda)}e^{\mathrm{i}\lambda t} \,d\lambda\right\rvert \le \int_{-\infty}^{\infty} \left\lvert \frac{1}{\alpha - \mathrm{i}\lambda} \partial_1\log{X(-\mathrm{i}\lambda)}e^{\mathrm{i}\lambda t} \right\rvert\,d\lambda =  \int_{-\infty}^{\infty}\left\lvert g(\lambda)\right\rvert \,d\lambda,
$$
and $g(\lambda)\in L^1(\mathbb{R})$, we can easily obtain that for $\alpha>0,\, t\rightarrow \infty$
$$
Q_{12}(t,t) \xrightarrow[\text{$t \to \infty$}]{} \frac{1}{2\pi}
\int_{-\infty}^{\infty} \frac{\alpha}{\alpha^2 + \lambda^2} \partial_1\log\left(1+\hat{K}_H(\lambda)\right)\,d\lambda.
$$
In fact $Q_{21}(t,t)$ holds the same limit.

At last for $Q_{22}$ we have the  direct calculation
$$
Q_{22}(t,t) \xrightarrow[\text{$t \to \infty$}]{}\frac{1}{2\pi}\int_{-\infty}^{\infty} 
\frac{1}{\alpha^2 + \lambda^2}\,d\lambda=\frac{1}{2\alpha}.
$$
\end{proof}

To prove the bounds of $R_{ij}^{(k)}$, the following lemma is given.
\begin{lem}
\label{lemR}
There exist constants $b \in (0,\frac{1}{2})$, $C > 0$ and $T_{min} > 0$ such that for all $s,t \ge T_{min}$,
\begin{equation}
\left\lvert R^{(k)}_{ij}(s,t) \right\rvert \le Cs^{-1/2}
\end{equation}
with $(i,j,k) \in \{(1,1,1),(1,2,1),(2,1,2),(2,2,1)\}$,
\begin{equation}
\left\lvert R^{(k)}_{ij}(s,t) \right\rvert \le Ct^{-1/2}   
\end{equation}
with $(i,j,k) \in \{(1,1,2),(1,2,2),(2,1,1),(2,2,2)\}$ and 
\begin{align}
&\left\lvert R^{(3)}_{11}(s,t) \right\rvert \le C(st)^{-b},\\
&\left\lvert R^{(3)}_{12}(s,t) \right\rvert \le Cs^{-b}t^{-1/2},\\
&\left\lvert R^{(3)}_{21}(s,t) \right\rvert \le Cs^{-1/2}t^{-b},\\
&\left\lvert R^{(3)}_{22}(t,s) \right\rvert \le C(st)^{-1/2}.
\end{align}
\end{lem}

\begin{proof}
Still by \cite{CK23}, the bound that $R^{(i)}_{11}(s,t)$ has already been provided.  Consider the case of $R^{(i)}_{12}(s,t)$, first of all we have 
\begin{equation}
\label{pf24}
\begin{aligned}
&\left\lvert Y(-\mathrm{i}\lambda)\partial_1\hat{R}_t(\mathrm{i}\lambda)\right\rvert 
\le \left\lvert (p_t(-\mathrm{i}\lambda) + q_t(-\mathrm{i}\lambda) + 1)\partial_1 \log Y(\mathrm{i}\lambda)\right\rvert 
+\\ &\left\lvert (p_t(\mathrm{i}\lambda)-q_t(\mathrm{i}\lambda))\partial_1 \log Y(\mathrm{i}\lambda)\right\rvert + 2\left\lvert \partial_1 p_t(\mathrm{i}\lambda)\right\rvert 
+2\left\lvert \partial_1 q_t(\mathrm{i}\lambda)\right\rvert ,
\end{aligned}
\end{equation}
where we use $\overline{Y(\mathrm{i}\lambda)} = Y(-\mathrm{i}\lambda)$. Then 
$ R^{(1)}_{12}(s,t) $ satisfies 
\begin{equation}
\label{new1}
\begin{aligned}
\left\lvert R^{(1)}_{12}(s,t)\right\rvert \le& \frac{1}{2\pi \alpha}
\int_{-\infty}^{\infty}\left\lvert 
\partial_1\hat{R}_s(\mathrm{i}\lambda)\log Y(-\mathrm{i}\lambda) \right\rvert\\
\le &C \int_{-\infty}^{\infty}\left(\left\lvert \partial_1 
p_s(\mathrm{i}\lambda)\right\rvert + \left\lvert \partial_1 
q_s(\mathrm{i}\lambda)\right\rvert + f_1(\lambda)
\left(\left\lvert p_s(\mathrm{i}\lambda) + \frac{1}{2}\right\rvert + \left\lvert q_s(\mathrm{i}\lambda) + \frac{1}{2} \right\rvert\right)  \right)\,d\lambda
\end{aligned}    
\end{equation}
where we define 
\begin{equation*}
f_1(\lambda) := \left\lvert \partial_1 \log X(\mathrm{i}\lambda) \right\rvert,  
\end{equation*}
then in view of (6.36) of \cite{CK23},
\begin{equation*}
f_1(\lambda) = 
\begin{aligned}
\begin{cases}
O(\log|\lambda|^{-1}),&\lambda \to 0,\\
O(|\lambda|^{1-2H_0}\log|\lambda|),&\lambda \to \pm \infty.
\end{cases}  
\end{aligned}
\end{equation*}
Thus $f_1 \in L^2(\mathbb{R})$. By estimate \eqref{eq:6.32} with $r=0$,
\begin{equation*}
\int_{-\infty}^{\infty}f(\lambda)\left\lvert p_s(\mathrm{i}\lambda) + \frac{1}{2}\right\rvert \,d\lambda \le C s^{-\frac{1}{2}}.
\end{equation*} 
The same estimate is valid for the rest of the integrals in eq.\eqref{new1}.  The bound of $R^{(2)}_{12}(s,t)$ can be given similarly, only change $f_1(\lambda)$ to the function $f_2(\lambda)$ defined by  
\begin{equation*}
f_2 := \left\lvert Y(\mathrm{i}\lambda)\partial_1 \log Y(-\mathrm{i}\lambda) \right\rvert    
\end{equation*}
and in view of \eqref{ff2} and \eqref{eq: big Lambda equal product}, we have 
\begin{equation*}
f_2(\lambda) = 
\begin{aligned}
\begin{cases}
O(|\lambda|^{\frac{1}{2}-H_0}\log|\lambda|^{-1}),&\lambda \to 0,\\
O(|\lambda|^{1-2H_0}\log|\lambda|),&\lambda \to \pm \infty,
\end{cases}  
\end{aligned}
\end{equation*}
Thus $f_2 \in L^2(\mathbb{R})$.

For $R^{(3)}_{12}(t,s)$,  consider
\begin{equation}
\label{pf26}
\begin{aligned}
|R^{(3)}_{12}(t,s)| &\le \frac{1}{\alpha} \int_{-\infty}^{\infty}|Y(-\mathrm{i}\lambda)\partial_1\hat{R}_t(\mathrm{i}\lambda)Y(\mathrm{i}\lambda)\overline{\hat{R}_s(\mathrm{i}\lambda)}|  \,d\lambda\\
&\le \frac{1}{\alpha} \left(\int_{-\infty}^{\infty}|Y(-\mathrm{i}\lambda)\partial_1\hat{R}_t(\mathrm{i}\lambda)|^2\,d\lambda\right) ^{1/2} \left(\int_{-\infty}^{\infty}|Y(\mathrm{i}\lambda)\hat{R}_s(\mathrm{i}\lambda)|^2\,d\lambda\right) ^{1/2}.
\end{aligned}
\end{equation}
Due to Lemma \ref{lempq}, for some $C > 0$,
\begin{equation}
\begin{aligned}
&\int_{-\infty}^{\infty} |Y(\mathrm{i}\lambda)\hat{R}_s(\mathrm{i}\lambda)|^2\,d\lambda \le\\
&2\int_{-\infty}^{\infty}\left\lvert (p_s(-\mathrm{i}\lambda) + q_s(-\mathrm{i}\lambda) + 1)\right\rvert^2  \,d\lambda\\
&+2\int_{-\infty}^{\infty}\left\lvert (p_s(\mathrm{i}\lambda) - q_s(\mathrm{i}\lambda) )\right\rvert^2  \,d\lambda\\
&\le C s^{-1}
\end{aligned}  
\end{equation}
by Lemma \ref{lempq} and Lemma 6.8 of \cite{CK23}, with $r \in (0,1)$ small enough,
\begin{equation}
\int_{-\infty}^{\infty}|Y(\mathrm{i}\lambda)\partial_1\hat{R}_s(\mathrm{i}\lambda)|^2\,d\lambda \le C t^{r-1}.
\end{equation}
Finally, we get
\begin{equation}
|R^{(3)}_{12}(t,s)| \le C s^{-1/2}t^{-b},
\end{equation}
where we define $b:=\frac{1}{2}-\frac{r}{2} \in (0,\frac{1}{2})$.

The proof of the growth estimates for $R^{(i)}_{21}(s,t),i=1,2,3$ and $R^{(i)}_{22}(s,t),i=1,2,3$ is something similar with the ones for $R^{(i)}_{12}(s,t)$.
\end{proof}

\subsection{Proof of Lemma \ref{key2}}
First we will calculate the second partial derivative with respect $b_t(X,\theta)$ they will be 
\begin{equation}
\begin{aligned}
\partial^2_1 b_t(X, \theta) &= \int_{0}^{t}\partial^2_1 g(t,t-s;H)\,dX_s + \alpha \int_{0}^{t}X_s \partial^2_1 g(t,t-s;H)\,ds,\\
&= \int_{0}^{t}\partial^2_1 g(t,t-s;H)\,dM_s,\\
\partial_1\partial_2 b_t(X,\theta) &= \partial_2\partial_1 b_t(X,\theta) = \int_{0}^{t} X_s \partial_1g(t,t-s;H)\,ds,\\
\partial^2_2 b_t(X, \theta) &= 0.
\end{aligned}    
\end{equation}
To prove for all sufficiently small $\delta > 0$ there exist constants $C > 0$ and $T_{min} > 0$ such that
\begin{equation}
\label{pf30}
\sup_{\left\lVert \theta - \theta_0\right\rVert \le \delta }  \mathbb{E} \left(\partial_i \partial_j b_t(X, \theta) \right) ^2 \le C, \quad \forall t \ge T_{min},
\end{equation}
Now we know 
\begin{equation}
\begin{aligned}
&\nabla^2 b_t(X, \theta)\\
&=\begin{pmatrix}
\partial_1 \partial_1 b_t(X, \theta) &  \partial_1 \partial_2 b_t(X, \theta)\,ds\\
\partial_2 \partial_1 b_t(X, \theta) & \partial_2 \partial_2 b_t(X, \theta), 
\end{pmatrix}\\
&=\begin{pmatrix}
\int_{0}^{t}\partial^2_1 g(t,t-s;H)\,dM_s  & \int_{0}^{t} X_s \partial_1 g(t,t-s;H)\,ds)\\
\int_{0}^{t} X_s \partial_1 g(t,t-s;H)\,ds & 0. 
\end{pmatrix},
\end{aligned}
\end{equation}
By proof of Lemma 5.2. of \cite{CK23}, we have
\begin{equation}
\label{pf33new}
\begin{aligned}
&\sup_{\left\lVert \theta-\theta_0\right\rVert\le \delta  }\mathbf{E}\left(\partial_1 \partial_1 b_t(X, \theta)\right)^2\\
=&\sup_{\left\lVert  H-H_0\right\rVert \le \delta  }\frac{1}{2\pi}\int_{-\infty}^{\infty}
\left\lvert  \partial_1^2 \hat{g}_t(\mathrm{i}\lambda;H)  \right\rvert^2\Lambda
(\mathrm{i}\lambda;H_0)\,d\lambda \le C,
\end{aligned}
\end{equation}
notice the true value of the parameter $\theta_0$, which determines the distribution of the sample $X^T$, then consider
\begin{equation}
\label{pf35new}
\begin{aligned}
&\mathbf{E}\left(\partial_1 \partial_2 b_t(X, \theta)\right)^2\\  
=&\mathbf{E}\left(\int_{0}^{t} \left(\partial_1 g(t,\cdot)* (-l(t,\cdot))\right)(t-s)\,dM_s 
\int_{0}^{t} \left(\partial_1 g(t,\cdot)* (-l(t,\cdot))\right)(t-s)\,dM_s \right)\\
=&\int_{0}^{t}\int_{0}^{t} \left(\partial_1 g(t,\cdot)* (-l(t,\cdot))\right)(t-s)
\left(\partial_1 g(t,\cdot)* (-l(t,\cdot))\right)(t-r)K_H(s-r)\,dsdr + \\
& \int_{0}^{t}\left(\partial_1 g(t,\cdot)* (-l(t,\cdot))\right)(t-s)
\left(\partial_1 g(t,\cdot)* (-l(t,\cdot))\right)(t-s)\,ds
\\
=& \frac{1}{2\pi}\int_{-\infty}^{\infty}\frac{1}{\alpha_0-\mathrm{i}\lambda}
\left(e^{\left(\mathrm{i}\lambda - \alpha_0\right)t}- 1\right)\frac{1}{\alpha_0+\mathrm{i}\lambda}
\left(e^{-\left(\mathrm{i}\lambda + \alpha_0\right)t}- 1\right)  \partial_1\hat{g}_t(\mathrm{i}\lambda;H)\partial_1 \overline{\hat{g}_t(\mathrm{i}\lambda;H)}
\Lambda(\mathrm{i}\lambda;H_0)\,d\lambda,\\
\end{aligned}
\end{equation}
it follows that
\begin{equation}
\label{pf34new}
\begin{aligned}
&\mathbf{E}\left(\partial_1 \partial_2 b_t(X, \theta)\right)^2=\\
\le&\frac{C_1}{2\pi}\int_{-\infty}^{\infty}\frac{1}{\alpha_0^2 + \lambda^2} \left\lvert \partial_1\hat{g}_t(\mathrm{i}\lambda;H)\right\rvert ^2 \Lambda(\mathrm{i}\lambda;H_0)\,d\lambda\\
\le&C_2\int_{-\infty}^{\infty}\left\lvert \partial_1\hat{g}_t(\mathrm{i}\lambda;H)\right\rvert ^2 \Lambda(\mathrm{i}\lambda;H_0)\,d\lambda\\
\le&C_3\left(\int_{-\infty}^{\infty}\left\lvert \partial_1 \frac{1}{X(\mathrm{i}\lambda;H)}\right\rvert ^2 \Lambda(\mathrm{i}\lambda;H_0)\,d\lambda 
+\int_{-\infty}^{\infty}\left\lvert \partial_1 \hat{R}_t(\mathrm{i}\lambda;H)\right\rvert ^2 \Lambda(\mathrm{i}\lambda;H_0)\,d\lambda  \right) ,
\end{aligned}
\end{equation}
for some $C_3$.

In view of eq.\eqref{ff1}
\begin{equation}
\label{ff4}
\left\lvert \partial_1 \frac{1}{Y(\mathrm{i}\lambda)}\right\rvert ^2 \Lambda(\mathrm{i}\lambda;H_0)=
\begin{aligned}
\begin{cases}
O(|\lambda|^{-2\delta}\log^4|\lambda|^{-1}),&\lambda \to 0,\\
O(|\lambda|^{2-4H}\log^2|\lambda|),&\lambda \to \pm \infty,
\end{cases}  
\end{aligned}
\end{equation}
this function is integrable on $\mathbb{R}$ for all suﬃciently small $\delta > 0$ so
\begin{equation}
\sup_{\left\lVert H-H_0\right\rVert\le \delta} \int_{-\infty}^{\infty}\left\lvert \partial_1 \frac{1}{Y(\mathrm{i}\lambda;H)}\right\rvert ^2 \Lambda(\mathrm{i}\lambda;H_0)\,d\lambda < C.   
\end{equation}

In view of eq.\eqref{ff3}, to prove 
\begin{equation}
\sup_{\left\lVert H-H_0\right\rVert\le \delta} \int_{-\infty}^{\infty}\left\lvert \partial_1 \hat{R}_t(\mathrm{i}\lambda;H)\right\rvert ^2 \Lambda(\mathrm{i}\lambda;H_0)\,d\lambda < C,       
\end{equation}
we only need to verify for all suﬃciently small $\delta > 0$, there exists positive constants $C$,$T_{\min}$ and $c$ such that
\begin{align}
\label{final1}
&\mathrm{I} = \int_{-\infty}^{\infty}\left\lvert \partial_1 \frac{1}{Y(\mathrm{i}\lambda;H)}\left(p_t(\mathrm{i}\lambda;H) + \frac{1}{2}\right) \right\rvert ^2\Lambda(\mathrm{i}\lambda;H_0)
\,d\lambda \le Ct^{-c},\\
\label{final2}
&\mathrm{II} = \int_{-\infty}^{\infty}\left\lvert  \frac{1}{Y(\mathrm{i}\lambda;H)}\partial_1 p_t(\mathrm{i}\lambda;H)  \right\rvert ^2\Lambda(\mathrm{i}\lambda;H_0)
\,d\lambda \le Ct^{-c},
\end{align}
for all $H$ such that $\left\lVert H-H_0\right\rVert\le \delta$ and all $t > T_{\min}$. 
Take any $r > 0$ small enough so that the assertion of Lemma \ref{lempq} holds.
Since we already have eq.\eqref{ff4},
it is clear that for all suﬃciently small $\delta > 0$, 
\begin{equation}
\left\lvert \partial_1 \frac{1}{Y(\mathrm{i}\lambda)}\right\rvert ^2 \Lambda(\mathrm{i}\lambda;H_0) \le C_1 |\lambda|^{-r},    
\end{equation}
then from Lemma \ref{lempq} $\mathrm{I} \le Ct^{-c}$ holds by setting $c = 1-r$.

On the other hand, from the previous presentation of $\Lambda(\lambda)$ we have 
$$
\begin{aligned}
\left\lvert  \frac{1}{Y(\mathrm{i}\lambda;H)} \right\rvert ^2\Lambda(\mathrm{i}\lambda;H_0) - 1
= \frac{\Lambda(\mathrm{i}\lambda;H_0)}{\Lambda(\mathrm{i}\lambda;H)}-1=
\begin{cases}
O(|\lambda|^{-2\delta}),&\lambda \to 0,\\
O(|\lambda|^{1-2H_0 + \delta}),&\lambda \to \pm \infty,
\end{cases}    
\end{aligned}
$$
by Lemma \ref{lempq}, we have 
\begin{equation}
\begin{aligned}
\mathrm{II} \le& \int_{-\infty}^{\infty}\left\lvert  \partial_1 p_t(\mathrm{i}\lambda;H)  \right\rvert ^2 
\left(\left\lvert \frac{1}{Y(\mathrm{i}\lambda;H)}\right\rvert ^2\Lambda(\mathrm{i}\lambda;H_0) -1 +1\right) 
\,d\lambda\\ 
\le& \int_{-\infty}^{\infty}\left\lvert  \partial_1 p_t(\mathrm{i}\lambda;H)  \right\rvert ^2 
\,d\lambda +\int_{-\infty}^{\infty}\left\lvert  \partial_1 p_t(\mathrm{i}\lambda;H)  \right\rvert ^2 
\left(\left\lvert \frac{1}{Y(\mathrm{i}\lambda;H)}\right\rvert ^2\Lambda(\mathrm{i}\lambda;H_0) -1 \right)\,d\lambda\\   
\le& Ct^{-1} + Ct^{-(1-r)}
\end{aligned}
\end{equation}

So by \eqref{final1} and \eqref{final2},
\begin{equation}
\sup_{\left\lVert \theta-\theta_0\right\rVert\le \delta}\mathbb{E}\left(\partial_1 \partial_2 b_t(X, \theta)\right)^2 \le C    
\end{equation}
for $\delta$ sufficiently small and some $t > T_{\min} $.

\end{document}